\documentclass[reqno]{amsart}
\usepackage{amsfonts}
\usepackage[active]{srcltx}
\usepackage{graphicx}
\usepackage{subfig}
\usepackage{amsmath}
\usepackage{amssymb}
\usepackage{amsthm}
\usepackage{algorithm}
\usepackage{algorithmic}
\usepackage{mathtools}
\usepackage{shuffle}
\usepackage{bm} 
\usepackage{upgreek}
\usepackage{graphicx} 
\usepackage{mathrsfs}
\usepackage{fancyhdr}
\usepackage{amsthm}
\usepackage{cases}
\usepackage{amsmath,amssymb,bm}
\usepackage{enumitem}
\usepackage{xcolor}

\usepackage{tikz-cd} 

\usepackage{enumitem}
\usepackage{wasysym}
\usepackage{mathscinet}
\usepackage{verbatim}

\usepackage{graphicx}
\usepackage[
bookmarks=true,         
bookmarksnumbered=true, 
colorlinks=true, pdfstartview=FitV, linkcolor=blue, citecolor=blue,
urlcolor=blue]{hyperref}
\usepackage{microtype}
\theoremstyle{plain}
\newtheorem{theorem}{Theorem}[section]
\newtheorem{prop}[theorem]{Proposition}

\newtheorem{corollary}[theorem]{Corollary}

\newtheorem{lemma}[theorem]{Lemma}
\newtheorem{problem}[theorem]{Problem}
\newtheorem{proposition}[theorem]{Proposition}

\renewenvironment{proof}[1][Proof]{\textbf{#1.} }{\ \rule{0.5em}{0.5em} \par }
\theoremstyle{remark}

\theoremstyle{definition}
\newtheorem{remark}[theorem]{Remark}

\def\RR{\mathbb{R}}
\def\EE{\mathbb{E}}

\def\XX{\mathbb{X}}

\def\la{{\lambda}}
\def\si{{\sigma}}

\def\al{{\alpha}}

\def\la{{\lambda}}

\def\th{{\theta}}
\def\th{{\theta}} 
\def\al{{\alpha}}

\def\cov{\hbox{Cov}}

\let\Section=\section
\def\section{\setcounter{equation}{0}\Section}

\title[Ergodic   Estimators  of  double exponential  Ornstein-Ulenbeck  process ]{Ergodic   Estimators  of 
double exponential Ornstein-Uhlenbeck processes} 
 
\date{} 

\author[Hu]{Yaozhong Hu}
 
\address{Department of Mathematical and Statistical Sciences \\
 University of Alberta at Edmonton \\
Edmonton,  Canada, T6G 2G1}
\email{  yaozhong@ualberta.ca}

\author[Sharma]{Neha Sharma} 
 \address{Department of Mathematical and Statistical Sciences \\
 University of Alberta at Edmonton \\
Edmonton,  Canada, T6G 2G1}
\email{neha2@ualberta.ca}
 
\keywords{Double exponential compound Poisson process, double exponential Ornstein-Uhlenbeck process, discrete time observation, ergodic theorem, ergodic estimators, strong consistency, 
central limit theorem, exact simulation. 
\newline 
This work is  supported by an NSERC discover fund and a start-up fund of University of Alberta. }
\begin{document}

\maketitle
\begin{abstract} 
The  goal of this paper is to construct ergodic estimators for  the parameters in the  double exponential   Ornstein-Uhlenbeck   process,   observed at discrete time instants with time step size $h$. The existence and uniqueness,  the strong consistency and the asymptotic normality of the  estimators are obtained for arbitrarily  fixed  time step size  $h$. A simulation method of the double exponential Ornstein-Uhlenbeck process  is proposed  and some numerical simulations are performed to demonstrate the effectiveness of the proposed estimators. 
\end{abstract}

\section{Introduction}
Let $(\Omega, \mathcal{F}, \mathbb{P})$ be a probability space with a right continuous family of increasing  $\sigma$-algebras  $(\mathcal{F}_t, t\geq 0)$ satisfying the usual condition
(\cite{meyer}).  The expectation on this probability space is denoted by $\EE$. Motivated by the recent successful applications  to finance (e.g. \cite{cai, hachmann} and references therein), we shall study in this work the parametric estimation problem for the double exponential Ornstein-Uhlenbeck process. To introduce this process  let   $(Y_n,n\geq 1)$ be a sequence of  independent  real valued  random variables  with the following    probability density function
\begin{equation}
f_{Y} (x)=p\eta e^{-\eta x}I_{[x\geq 0]}+q\varphi e^{\varphi x}I_{[x< 0]}\,, \label{e.1.1} 
\end{equation}  
where the parameters $p $, $q $, $\eta $,  $\varphi $ are positive and $p+q=1$. Let ${N_t}$ be the Poisson process with rate $\lambda>0$,  independent of $\{Y_i,i=1,2,\ldots\}$.  
Then  $Z_t=\sum_{i=1}^{N_t}Y_i $ is called the double exponential    compound Poisson process,  which is a particular   L\'evy process. The stochastic calculus with respect to this process falls in the framework of   the stochastic calculus for general 
L\'evy processes.  For more details we refer to \cite{Sato}     whose results will be used       freely.  
 
 Let us consider the  following  double exponential  Ornstein-Ulenbeck 
 process given by the following Langevin equation driven by the double exponential compound Poisson process $Z_t$. 
\begin{equation}
    dX_t =-\theta X_t dt 
    +\sigma dZ_t \,, \quad  t\in[0,\infty),\hspace{5mm}X_0=x_0\,. \label{e.1.2}  
\end{equation}
Of course,  this equation is interpreted as its integral form: 
\begin{equation}
    X_t =x_0-\theta\int_0^t X_s ds   +\sigma Z_t \,. 
\end{equation}
 This    process $X_t$ depends on  the parameters     $\theta$, $\sigma$, $p$ (or  $q$), $\eta$,  $\la$,  and  $\varphi$.  We assume that the process  
  $\{X_t;t\geq 0\}$ can be observed at  discrete time instants $t_j=jh$,  where $h>0$ is some observation 
  time interval. We want to use the    observation data set $\{X_{t_j}; j=1,2,\ldots n\}$ to estimate  
 the parameters    $\theta$, $\sigma$, $p$,  $\eta$,  $\la$, and  $\varphi$. To construct such estimators, we shall use the ergodic theorem $\lim_{n\rightarrow\infty}
 \frac1n \sum_{j=1}^\infty f(X_{t_j})=\int_\RR f(x) \mu(dx)$, where $\mu$ is the limiting distribution of $x_t$.
 One may think  that with appropriate choices of different $f$ we may be able to find all the parameters. However, 
the limiting distribution depends on the parameters in such a way (e.g. \eqref{e.2.7}) that one cannot decouple them. For this reason and motivated by 
\cite{Hu}, we get involved the ergodic theorem of the form  $\lim_{n\rightarrow\infty}
 \frac1n \sum_{j=1}^\infty g(X_{t_j}, X_{t_{j+1}})=\int_\RR g(x,y) \nu(dx, dy)$, where $\nu(dx,dy)$ is the limiting  distribution of $(X_t, X_{t+h})$. After finding the  distribution of $\mu$ and $\nu$ we use the moment functions to obtain appropriate equations so that the ergodic estimators satisfy. 

An immediate problem after the obtention of the   equations to identify the parameters is the 
well-posedness of such system: the existence, local uniqueness and global uniqueness of the system. We shall also address this   elementary and challenging 
 problem assuming $\sigma=\lambda=1$.  We shall prove that when the sample size is sufficiently large we shall have the existence and uniqueness of a local solution. With a further treatment, we reduce the  problem of global uniqueness  of the system to  a problem of finding zero for a real valued function of one variable, where the mean value theorem can be used.
The strong consistency and asymptotic normality of our ergodic estimators are also given. 
 
 To validate our approach we propose   an  exact decomposition simulation algorithm for our double exponential Ornstein-Ulenbeck  process. This algorithm  allows us to write the distribution of $X_{t+h}$ given $X_t$ as a sum of deterministic function and a mixed compound Poisson process. After discussing the algorithm we simulate the data from $\eqref{e.1.1}$ assuming some given values  of $\theta$, $p$,  $\eta$,  and  $\varphi$. Then we apply the estimators to estimate these parameters. The numerical results show that our estimators converge fast to the true parameters.   
%

The paper is organized as follows. 
In Section 2, we give some    preliminaries and some  basic results for our double exponential 
Ornstein-Uhlenbeck  process. We also obtain the explicit form of  the  characteristic functions of 
 limiting distributions $\mu$ and $\nu$ mentioned earlier. In Section 3, we construct the ergodic estimators for all the parameters in the double exponential 
 Ornstein-Uhlenbeck process. The local existence, uniqueness and the global uniqueness of the system of equations determining these ergodic estimators are discussed. In Section 4, the joint asymptotic normality of the the estimators is obtained. In Section 5 , we discuss the exact decomposition algorithm for simulating the process. In Section 6 we perform some numerical simulations to validate our results which demonstrate the effectiveness of our estimators. 
Section 7 contains the computation of a covariance matrix appeared in our theorem. 

\section{Preliminaries}
The equation $\eqref{e.1.2}$ has a unique  solution 
given by 
\begin{equation}
    X_t =e^{-\theta t}x_0+\sigma \int_0^te^{-\theta (t-s)}dZ_s  \,.  \label{e.2.1}
\end{equation}
If  $\theta>0$, then  the double exponential Ornstein-Uhlenbeck   process $X_t$   converges in law to the  random variable $ \XX_o =\sigma \int_0^{\infty}e^{-\theta s}dZ_s$.  If the initial condition $X_0$ has the   law   of $  \XX_o$, namely, if the process starts at the stationary distribution 
and if $X_0$ is independent of the process $Z_t$,  then  $X_t$ is a stationary  process.
It is well-known from   \cite[Theorem 17.5]{Sato}) that   the double exponential process  $\{X_t,  t\ge 0\}$  is ergodic.
Namely, we have the following result from \cite[Theorem 8.1]{Meyn}.
  
\begin{proposition}\label{p.2.1} 
Let  $f:\mathbb{R}\rightarrow  \mathbb{R}$ be measurable such that $\EE|f(\XX_o)|<\infty$.
Then for any initial condition $x_0\in \RR$ and for any $h\in \RR_+$,  we have  (denoting $t_j=jh$) 
\begin{equation}
\lim_{n\rightarrow\infty}\frac1n \sum_{j=1}^nf(X_{t_j})=\EE(f(\XX_o))\hspace{10mm}a.s. \label{eq1}
\end{equation}
\end{proposition}
The explicit form  of the 
distribution of $\XX_o$ is hard to obtain.
So,  it is hard to  compute $\EE(f(\XX_o))$ for general $f$.  But when  $f$ has  some particular form, namely, when $f(x)=e^{\iota \xi x}$,  then the 
computation of 
  $\EE(f(\XX_o))$  is much  simplified. 
\subsubsection{Evaluation of the limiting characteristic functions}
\begin{prop}\label{p.2.2}
Let   $\{Z_t\}$ be the double exponential compound Poisson process 
and let $0<s<t<\infty$. Then 
for any real valued continuous function $g(u)$ on $[s,t]$    we have 
\begin{equation}
\EE\Big[\exp{\Big(i z\int_s^tg(u)dZ_u(\omega)\Big)}\Big]=\exp{\Big[ \int_s^t\Psi(g(u)z)du\Big]}\,, \quad \forall \ 
z\in \RR \,, \label{e.2.3} 
\end{equation}
where
\begin{equation}
\Psi(z)= \log \hat P_{Z_1}(u) = \log \EE \Big[e^{iuZ_1}\Big]=   \lambda   \int_\RR  e^{iuy} f_{Y} ( y)dy
    -\la      \label{e.2.4}
\end{equation}  
with $f_Y$ being given by \eqref{e.1.1}.  
\end{prop}
\begin{proof} 
We  follow  the idea of  \cite[Section 17]{Sato}. 
Let us first compute the characteristic function of $Z_t$. 
\begin{equation*}
\begin{split}
 \hat P_{Z_t}(u) := &  \EE \Big[e^{iuZ_t}\Big] =\EE\Big[e^{iu\sum_{j=1}^{N_t}Y_j}\Big]\\
     =&\sum_{n=0}^{\infty} E\Big[e^{iu\sum_{j=1}^{n}Y_j}|N_t=n\Big] P(N_t=n)\\
    &=\sum_{n=0}^{\infty}\frac{(\lambda \cdot t )^n}{n!}e^{-\lambda}\Big(\EE(e^{iuY_1})\Big)^n\\
     =&e^{[\lambda t \EE(e^{iuY_1})-\la ]}=
    \exp \left[  \lambda t \int_\RR  e^{iuy} f_{Y} ( y)dy
    -\la \right]\,,
\end{split}
\end{equation*}
where $f_Y(y)$ is the double exponential density  defined by \eqref{e.1.1}. 
When $t=1$   we have \eqref{e.2.4}. 
\end{proof}
Now we are going to compute the characteristic function of the limiting distribution of $\XX_o$. From   Equation \eqref{e.2.1} and 
Proposition \ref{p.2.2}  it follows 
\begin{equation}
\begin{split}
    \EE[e^{i uX_t}]=
    &\exp{\Big[ie^{-\theta t} x_0u+\int_0^t\Psi(\sigma e^{-\theta s}u)ds\Big]}\\
    =& \exp{\Bigg[ie^{-\theta t} x_0u+
    \la \int_0^t \left[\int_\RR  e^{i \sigma e^{-\theta s}uy} f_{Y} ( y)dy-1 \right] 
     ds\Bigg]}\\
     =& \exp{\Bigg[ie^{-\theta t} x_0u+
    \la I_{1,t}\Bigg]}  \,, 
     \end{split} \label{e.2.5} 
\end{equation}
where $ I_{1,t}=\int_0^t [I_{2,s}-1]ds$ and $I_{2,s}$ is defined and computed as follows. 
\begin{eqnarray*}
I_{2,s}&=& \int_\RR  e^{i \sigma e^{-\theta s}uy} f_{Y} ( y)dy\\
&=& \int_\RR  e^{i \sigma e^{-\theta s}uy} 
\left[p\eta e^{-\eta y}I_{[y\geq 0]}+q\varphi e^{\varphi y}I_{[y< 0]} \right] dy\\
&=& p \eta \int_0^\infty  e^{i \sigma e^{-\theta s}uy} 
  e^{-\eta y} dy+q\varphi \int_{-\infty}^0 
 e^{i \sigma e^{-\theta s}uy}   e^{\varphi y}    dy\\
 &=& \frac{p\eta}{\eta-i\sigma u e^{-\theta s }}+
 \frac{q\varphi }{\varphi +i\sigma u e^{-\theta s }}\,, 
\end{eqnarray*}
where in the above second identity we used the explicit 
form of $f_Y$ given by \eqref{e.1.1}.  
 
Thus 
\begin{align}
I_{1,t}= &  \int_0^t[I_{2,s}-1]ds 
\nonumber \\
     = & \int_0^t \Big(\frac{p\eta}{\eta-i\sigma e^{-\theta s}u}+\frac{q\varphi}{\varphi+i\sigma e^{-\theta s}u}-1\Big)ds \nonumber \\
     = &\frac{p }{\theta}\ln\Big(\frac{\eta-i\sigma e^{-\theta t} u}{\eta-i\sigma u}\frac{1}{ e^{-\theta t}}\Big)+\frac{q }{\theta}\ln\Big(\frac{\varphi+i e^{-\theta t}\sigma u}{\varphi+i\sigma u}\frac{1}{ e^{-\theta t}}\Big)-  t  
  \nonumber   \\     
    =&  \ln\left[ \Big(\frac{\eta-i\sigma e^{-\theta t} u}{\eta-i\sigma u}\frac{1}{ e^{-\theta t}}\Big)^{\frac{p }{\theta}}
   \cdot \Big(\frac{\varphi+i e^{-\theta t}\sigma u}{\varphi+i\sigma u}\frac{1}{ e^{-\theta t}}\Big)^{\frac{q }{\theta}}
   \cdot  e^{ - t}  \right]\nonumber \\
   =&  \ln\left[ \Big(\frac{\eta-i\sigma e^{-\theta t} u}{\eta-i\sigma u}\ \Big)^{\frac{p }{\theta}}
   \cdot \Big(\frac{\varphi+i e^{-\theta t}\sigma u}{\varphi+i\sigma u} \Big)^{\frac{q }{\theta}}
      \right] \,, \label{e.2.6} 
\end{align}
where in the above last identity, we used $p+q=1$.  
Consequently,  we 
have as $t\rightarrow \infty$  
\begin{equation*}
 \lim_{t\to \infty} I_{1,t}= \ln\left[ \Big(\frac{\eta }{\eta-i\sigma u}\ \Big)^{\frac{p }{\theta}}
   \cdot \Big(\frac{\varphi }{\varphi+i\sigma u} \Big)^{\frac{q }{\theta}}
      \right] \,. 
 \end{equation*}  
 This combined with \eqref{e.2.5} yields   
\begin{align}
\lim_{t\rightarrow\infty}\EE[e^{i\langle u,X_t\rangle}]
&=\Big(\frac{\eta}{\eta-i\sigma u}\Big)^{\frac{p\lambda}{\theta}}\Big(\frac{\varphi}{\varphi+i\sigma u}\Big)^{\frac{q\lambda}{\theta}}\,. \label{e.2.7}
\end{align}
In other words, we have
\begin{eqnarray}
\EE  \left[ e^{i\langle u,\XX_o\rangle}\right] 
=\lim_{t\rightarrow\infty}\EE[e^{i\langle u,X_t\rangle}]=\Big(\frac{1}{1-i  u\frac{\sigma}{\eta}}\Big)^{\frac{p\lambda}{\theta}}\Big(\frac{1}{1+i  u\frac{\varphi}{\eta} }\Big)^{\frac{q\lambda}{\theta}} 
 \,. \label{e.2.8} 
\end{eqnarray}  
The probability distribution function of $\XX_o$ is   uniquely determined by the above characteristic function 
\eqref{e.2.8}.  This formula also means that the invariant random variable
$\XX_o$ depends  on $\frac{\sigma}{\eta}, \frac{\varphi}{\eta}, \frac{\la}{\th}$  and then we cannot  separate the parameters $\theta$, $\sigma$,  $\eta$, $\varphi$,  $\lambda$ and  $p$.
 
Motivated by the works of \cite{Hu, mehdi} we use the multi-time ergodic theorem to find more parameters.
Our theoretical basis is the following
general ergodic result, which is a consequence of      \cite[Theorem 1.1]{1961}. 
\begin{equation}
\lim_{n\rightarrow\infty}\frac{1}{n}\sum_{j=1}^ng(X_{t_j}, X_{t_{j}+h} )=\EE\left[g\Big(\XX_0, \XX_h\Big)\right]   \label{e.2.9}
\end{equation} 
 where $\XX_t$ satisfies the Langevin equation $\eqref{e.1.2}$ with
the initial condition  $\XX_0=\XX_o$, namely, $d\XX_t = -\theta \XX_tdt   +\sigma dZ_t$ and $\XX_0$ has  the invariant measure given by \eqref{e.2.8}. The right hand side of \eqref{e.2.9} is hard to compute for  general $g$. So we shall compute 
\begin{equation}
\lim_{n\rightarrow\infty}\frac{1}{n}\sum_{j=1}^n\exp{[iuX_{t_j}+ivX_{t_{j}+h}]}=\EE\left[\exp \Big( iu\XX_0+iv\XX_h\Big)\right]   \label{eq2}
\end{equation} 
for arbitrary $u, v\in\mathbb{R}$.  
In fact, we shall evaluate  the above quantity by evaluating    $\lim_{t\rightarrow\infty} \EE[e^{i( uX_t+vX_{t+h})}]$.  We shall still use the formula \eqref{e.2.3} to do our computations. As we see we can assume $X_0=0$.  Thus,  
\begin{eqnarray*}
&&X_t(\omega)=\sigma \int_0^te^{-\theta (t-s)}dZ_s(\omega) \,;\quad X_{t+h}(\omega)
=\sigma \int_0^{t+h}e^{-\theta (t+h-s)}dZ_s(\omega)\,.
\end{eqnarray*} 
Therefore,
\begin{align}
    uX_t(\omega)+vX_{t+h}(\omega)
    =&\sigma \int_0^t(ue^{-\theta (t-s)}+ve^{-\theta (t+h-s)})dZ_s\nonumber \\
   &\qquad +\sigma\int_t^{t+h}ve^{-\theta (t+h-s)}dZ_s\,. 
\end{align}
Because of the independent increment property of the double exponential compound Poisson process $Z_t$, we have 
\begin{align}
  \EE\Big[\exp\Big(iuX_{t}+ivX_{t+h}\Big)\Big]   =
  &\EE\Big[\exp\Big(i\int_0^t\sigma e^{-\theta (t-s)}(u+ve^{-\theta h})dZ_s\Big)\Big]\nonumber\\
  &\qquad  \EE\Big[i\int_t^{t+h}\sigma ve^{-\theta (t+h-s)}dZ_s\Big]\nonumber\\
  =:& I_{3, t}\cdot I_{4,t}\,,  
\end{align}
where $I_{3, t}$ and $ I_{4,t}$ denote the above first and second expectations. 
Similar to \eqref{e.2.6}, we have 
\begin{align}
   I_{3,t} = 
   &\exp\Big[\frac{p\lambda}{\theta}\ln\Big(\frac{\eta-i\sigma( e^{-\theta t} u+ve^{-\theta (t+h)})}{\eta-i\sigma (u+ve^{-\theta h})}\frac{u+ve^{-\theta h}}{ ue^{-\theta t}+ve^{-\theta (t+h)}}\Big)\nonumber \\
  & \qquad\quad   +\frac{q\lambda}{\theta}\ln\Big(\frac{\varphi+i\sigma (e^{-\theta t} u+ve^{-\theta (t+h)})}{\varphi+i\sigma (u+ve^{-\theta h})}\nonumber \\
  &\qquad \qquad\qquad  \cdot \frac{u+ve^{-\theta h}}{ e^{-\theta t} u+ve^{-\theta (t+h)})}\Big)-\ln(e^{\lambda t})\Big]
\end{align} 
and 
\begin{align}
  I_{4,t}   &=\Big(\frac{\eta-i\sigma v}{\eta-i\sigma e^{-\theta h} v}\Big)^{\frac{p\lambda}{\theta}}\Big(\frac{\varphi+i e^{-\theta h}\sigma v}{\varphi+i\sigma v}\Big)^{\frac{q\lambda}{\theta}}\,. 
\end{align}
It may be a bit strange to see that $I_{4,t}$ is independent of $t$. But this is because of the
independent increment property of the process $Z_t$.
In fact, we see easily that $\int_t^{t+h}\sigma ve^{-\theta (t+h-s)}dZ_s$ has the same law as that of  $\int_0^{ h}\sigma ve^{-\theta (h-s)}dZ_s$.  
It is easy to verify 
\begin{equation}
\begin{split}
 \lim_{t\to\infty} 
 I_{3,t}  & =\Big(\frac{\eta}{\eta-i\sigma (u+ve^{-\theta h})}\Big)^{\frac{p\lambda}{\theta}}
\Big(\frac{\varphi}{\varphi+i\sigma (u+ve^{-\theta h})}\Big)^{\frac{q\lambda}{\theta}} \,.  
\end{split}
\end{equation}
Hence, we have 
\begin{equation}
\begin{split}
\EE\left[ \exp\Big(iu\XX_0+iv\XX_h\Big)\right]
=&   \lim_{t\rightarrow\infty}\EE\Big[\exp\Big(iuX_{t}+ivX_{t+h}\Big)\Big]  \\
  =&\Big(\frac{\eta}{\eta-i\sigma (u+ve^{-\theta h})}\Big)^{\frac{p\lambda}{\theta}}
\Big(\frac{\varphi}{\varphi+i\sigma (u+ve^{-\theta h})}\Big)^{\frac{q\lambda}{\theta}}   \\
&\qquad \cdot\Big(\frac{\eta-i\sigma e^{-\theta h} v}{\eta-i\sigma v}\Big)^{\frac{p\lambda}{\theta}}\Big(\frac{\varphi+i e^{-\theta h}\sigma v}{\varphi+i\sigma v}\Big)^{\frac{q\lambda}{\theta}}\,. 
\label{e.2.16}
\end{split}
\end{equation}
We summarize \eqref{eq1}, \eqref{e.2.8},
\eqref{eq2}, \eqref{e.2.16} as the following theorem.
\begin{theorem} \label{t.2.3} 
Let $X_t$  be the double exponential 
Ornstein-Uhlenbeck process with  initial condition $x_0\in \RR$.  Then for any $h\in \RR_+, u,v\in \RR$,  we have  almost surely   (denoting $t_j=jh$) 
\begin{equation}
\left\{
\begin{split}
 &\lim_{n\rightarrow\infty}\frac{1}{n}\sum_{j=1}^n e^{iu X_{t_j}} 
= \Big(\frac{\eta }{\eta-i  u \sigma }\Big)^{\frac{p\lambda}{\theta}}\Big(\frac{\eta }{\eta+i  u \varphi }\Big)^{\frac{q\lambda}{\theta}}   \\ 
& \lim_{n\rightarrow\infty}\frac{1}{n}\sum_{j=1}^n\exp{[iuX_{t_j}+ivX_{t_{j}+h}]} \\ 
 &\qquad\qquad  \  =\Big(\frac{\eta}{\eta-i\sigma (u+ve^{-\theta h})}\Big)^{\frac{p\lambda}{\theta}}
\Big(\frac{\varphi}{\varphi+i\sigma (u+ve^{-\theta h})}\Big)^{\frac{q\lambda}{\theta}}    \\
 &\qquad \qquad \qquad\qquad    \cdot\Big(\frac{\eta-i\sigma e^{-\theta h} v}{\eta-i\sigma v}\Big)^{\frac{p\lambda}{\theta}}\Big(\frac{\varphi+i e^{-\theta h}\sigma v}{\varphi+i\sigma v}\Big)^{\frac{q\lambda}{\theta}}\,. 
\end{split}  \right.   \label{e.2.17} 
\end{equation} 
\end{theorem}

\section{Estimation of the parameters $\eta$, $\theta$, $\varphi$ and $p$}
Assume now that the double exponential 
Ornstein-Uhlenbeck process can be observed at discrete time so that the observation data $\{X_{t_j}, j=1, \cdots, n\}$
are available  to us,  where $t_j=j h$
for some given observation  time interval length $h$. 
Presumably Theorem \ref{t.2.3} can be used to estimate all  the parameters $\eta$, $\theta$, $\varphi$, $\la$, $\sigma$,  and $p$ by  replacing the limits in 
\eqref{e.2.17}  by their   the empirical characteristic functions $\hat{\Psi}_n(u)$  and $\hat{\Theta}_n(u,v)$ defined as  follows 
\begin{equation}
 \left\{
 \begin{split} 
&  \hat{\Psi}_{1,n}(u, v):=\frac{1}{n}\sum_{j=1}^n\exp{iuX_{t_j}}\,;  \\
&\hat{\Psi}_{2,n}(u,v)=\frac{1}{n}\sum_{j=1}^n\exp({iuX_{t_j}+ivX_{t_j+h}} )\,. 
\end{split}\right. 
\end{equation}
For any given pair $(u, v)$  although $\hat{\Psi}_{1,n}(u, v)$ depends only on $u$
we write it as a function of $u,v$ for convenience.
Since we have $6$ parameters, it may be possible for us to choose appropriately $6$ pairs of $(u_k, v_k)$ such that the $6$ parameters can be determined by 
\begin{equation}
\left\{
\begin{split}
 &  \Big(\frac{\eta }{\eta-i  u_k\sigma }\Big)^{\frac{p\lambda}{\theta}}\Big(\frac{\eta }{\eta+i  u_k \varphi }\Big)^{\frac{q\lambda}{\theta}}
  =\hat{\Psi}_{1,n}(u_k, v_k) \,, \quad k=1, \cdots, m,\\ 
&    \Big(\frac{\eta}{\eta-i\sigma (u_k+v_ke^{-\theta h})}\Big)^{\frac{p\lambda}{\theta}}
\Big(\frac{\varphi}{\varphi+i\sigma (u_k+v_ke^{-\theta h})}\Big)^{\frac{q\lambda}{\theta}}     \\
 &\qquad \qquad     \cdot\Big(\frac{\eta-i\sigma e^{-\theta h} v_k}{\eta-i\sigma v_k}\Big)^{\frac{p\lambda}{\theta}}\Big(\frac{\varphi+i e^{-\theta h}\sigma v_k}{\varphi+i\sigma v_k}\Big)^{\frac{q\lambda}{\theta}}=\hat{\Psi}_{2,n}(u_k, v_k)\,, \\
 &\qquad\qquad\qquad \quad k=m+1, \cdots, 6 \,, 
\end{split}  \right.   \label{e.3.2} 
\end{equation} 
where $m$ is some integer between $1$ and $6$. 
For any given pair $(u, v)$,
the empirical   characteristic  
functions  $\hat{\Psi}_{1,n}(u, v)$
and   $\hat{\Psi}_{1,n}(u, v)$  are known since we have the available observation data.  
Thus \eqref{e.3.2} is a 
system of function equations on the parameters 
$\eta$, $\theta$, $\varphi$, $\la$, $\sigma$,  and $p$. 
We believe that with appropriate choice of $(u_k, v_k)$ we should be able to use \eqref{e.3.2} to estimate all the above  six parameters. However, it is still difficult for us 
to argue if this  system of equations have a global unique solution or not although this  system  of nonlinear function equations \eqref{e.3.2}  is explicit and appears to be quite simple as well. Since we want to deal with the global uniqueness of the system \eqref{e.3.2}, we shall assume 
$\la=\sigma=1$ so that we now have only four parameters: $\eta$, $\theta$, $\varphi$,  and $p$. 
Supposedly we should be able to choose four different values of $(u_k, v_k)$ so that we
obtain a system of four equations  for the four unknowns. However, it is still difficult to argue the global uniqueness for the obtained system. 
So we are proposing an alternative method. 
Since \eqref{e.2.17} holds true for all $(u, v)\in \RR$ we can obtain explicit formulas  for the moments and then we use the moments to identity the parameters.  
Since $\EE|\XX_0|^m<\infty$ and  $\EE|\XX_0\XX_h|^m<\infty$ for all $m$      we know 
that \eqref{eq1}       and \eqref{e.2.9}  hold true for moment functions, in particular, we shall choose $f=x, x^2, x^3, g(x,y)=xy$.   
Thus the system of four equations we choose to obtain the  
estimators for  $\eta$, $\theta$, $\varphi$,  and $p$ are  
\begin{equation}
\left\{ \begin{split}
& \EE[\XX_o]=\mu_{1,n}, \quad\hbox{where}\  \mu_{1,n}:= \frac{1}{n}\sum_{j=1}^nX_{t_j} \,,\\
 &   \EE[\XX_o^2]=\mu_{2,n}, \quad\hbox{where}\ \mu_{2,n}:= \frac{1}{n}\sum_{j=1}^nX_{t_j}^2 \,,  \\ 
&    \EE[\XX_o^3]=\mu_{3,n}, \quad\hbox{where}\ \mu_{3,n}:=\frac{1}{n}\sum_{j=1}^nX_{t_j}^3 \,,\\
&   \EE[\XX_o\XX_h]= \mu_{4,n}, \quad\hbox{where}\ \mu_{4,n}:= \frac{1}{n}\sum_{j=1}^nX_{t_j}X_{t_j+h} \,. 
\end{split} \right.  \label{e.3.4} 
\end{equation}  
The right hand sides of \eqref{e.3.4} (namely, $\mu_{i, n}, i=1, 2, 3, 4$)  
are known from the discrete time observations of the double exponential 
Ornstein-Uhlenbeck process $X_t$. The left hand sides of \eqref{e.3.4} are functions of the parameters $\eta$, $\theta$, $\varphi$,  and $p$. We need first to find out    how they depend on the  four  parameters explicitly and then solve this system to construct the ergodic estimators 
$\hat \eta_n$, $\hat\theta_n$, $\hat \varphi_n$,  and $\hat p_n$for the parameters.  Let us also emphasize that \eqref{e.3.4} are not equations for the true parameters but they are equations for the ergodic estimators. 

Now let us find the explicit forms for the left hand sides of \eqref{e.3.2}. Let 
$
\rho=\frac{\sigma}{\eta}\quad \hbox{and}\quad \xi=\frac{\sigma}{\varphi}\,. 
$\  
From the identities  
\eqref{e.2.8} and \eqref{e.2.16},
we see   by the expression of moments through characteristic function (e.g. Corollary 1 to Theorem 2.3.1 in \cite{Eugene}) 
\begin{equation}
\left\{ \begin{split}
 \EE[\XX_o]&=\frac{1}{i}\frac{\partial }{\partial u}\EE  [ e^{i\langle u,\XX_o\rangle}]\bigg|_{u=0}\\ 
&=\frac{1}{i}\frac{\partial }{\partial u}\Big(\frac{1 }{1-i  u \rho }\Big)^{\frac{p\lambda}{\theta}}\Big(\frac{1 }{1+i  u \xi }\Big)^{\frac{q\lambda}{\theta}}\bigg|_{u=0} \\ 
&=\frac{\lambda}{\theta}\Big[p\rho-q\xi\Big]\,; \\
\EE[\XX_o^2]&= \frac{1}{i^2}\frac{\partial^2 }{\partial u^2}\EE  [ e^{i\langle u,\XX_o\rangle}]\bigg|_{u=0}\\   
&=\frac{\lambda}{\theta}\Big[p\rho^2+q\xi^2\Big]+\frac{\lambda}{\theta}\Big[p\rho-q\xi\Big]^2\\
&=\frac{\lambda}{\theta}\Big[p\rho^2+q\xi^2\Big]+\EE[\XX_o]^2\,;\\
\EE[\XX_o^3]&= \frac{1}{i^3}\frac{\partial^3 }{\partial u^3}\EE  [ e^{i\langle u,\XX_o\rangle}]\bigg|_{u=0}\\
&=\frac{2\lambda}{\theta}\Big[p\rho^3-q\xi^3\Big]+\Big(\frac{\lambda}{\theta}\Big[p\rho^2+q\xi^2\Big]+\frac{\lambda}{\theta}\Big[p\rho-q\xi\Big]^2\Big)\Big(\frac{\lambda}{\theta}\Big[p\rho-q\xi\Big]\Big)\\
&+2\frac{\lambda}{\theta}\Big[p\rho-q\xi\Big]\Big(\frac{\lambda}{\theta}\Big[p\rho^2+q\xi^2\Big]\Big)\\
&=\frac{2\lambda}{\theta}\Big[p\rho^3-q\xi^3\Big]+\EE[\XX_o^2]\EE[\XX_o]+2\EE[\XX_o](\EE[\XX_o^2]-\EE[\XX_o]^2)\,;\\
\EE[\XX_o\XX_h]&=   \frac{1}{i^2}\frac{\partial }{\partial v}\frac{\partial }{\partial u}\EE\left[ \exp\Big(iu\XX_0+iv\XX_h\Big)\right]\bigg|_{u=0,v=0}\\
&=\frac{1}{i^2}\frac{\partial }{\partial v}\frac{\partial }{\partial u}\Big(\frac{1}{1-i\rho (u+ve^{-\theta h})}\Big)^{\frac{p\lambda}{\theta}}
\Big(\frac{1}{1+i\xi (u+ve^{-\theta h})}\Big)^{\frac{q\lambda}{\theta}}    \\
 &\qquad\qquad\qquad    \cdot\Big(\frac{1-i\rho e^{-\theta h} v}{1-i\rho v}\Big)^{\frac{p\lambda}{\theta}}\Big(\frac{1+i e^{-\theta h}\xi v}{1+i\xi v}\Big)^{\frac{q\lambda}{\theta}}\bigg|_{u=0,v=0}\\
 &=e^{-\theta h}\frac{\lambda}{\theta}\Big[p\rho^2+q\xi^2\Big]+\frac{\lambda^2}{\theta^2}\Big[p\rho-q\xi\Big]^2\\
 &=e^{-\theta h}\frac{\lambda}{\theta}\Big[p\rho^2+q\xi^2\Big]+\EE[\XX_o]^2
\end{split} \right. \,.  \label{e.3.5} 
\end{equation} 
An elementary simplification yields  (noticing $\la=1$)
\begin{align}
&\frac{1}{\theta}\Big[p\rho-q\xi\Big]
     =\mu_{1,n}\,, \label{e.3.6}\\
&\frac{1}{\theta}\Big[p\rho^2+q\xi^2\Big]=
     \mu_{2n} -\mu_{1,n}^2 \,,  \label{e.3.7}\\
&    \frac{2 }{\theta}\Big[p\rho^3-q\xi^3\Big]
= \mu_{3,n}-\mu_{2,n} \mu_{1,n}-2\mu_{1,n}(\mu_{2,n}-\mu_{1,n}^2)\,, \label{e.3.8}\\  
    &\frac{1}{\theta}e^{-\theta h}\Big[p\rho^2+q\xi^2\Big]=\mu_{4,n}-\mu_{1,n}^2\,. \label{e.3.9}
\end{align}
Thus we have the explicit form \eqref{e.3.6}-\eqref{e.3.9} for \eqref{e.3.4}.  Now we want to solve this system of function equations (e.g. \eqref{e.3.6}-\eqref{e.3.9}). 
Dividing \eqref{e.3.7} by \eqref{e.3.9} gives 
\begin{equation}
   \hat{ \theta}_n=\frac{1}{h}\ln \Big(\frac{\mu_{2,n}-\mu_{1,n}^2}{\mu_{4,n}-  \mu_{1,n} ^2}\Big) \,. 
   \label{e.3.10} 
\end{equation}  
Now we use  the three equations 
\eqref{e.3.6}-\eqref{e.3.8} to solve for the remaining three unknowns 
$p, \rho, \xi$ (noticing $q=1-p$). 
Denote 
\begin{equation}  
\left\{
\begin{split} 
f_1 &=\hat{ \theta}_n \mu_{1,n}\,,    \\
    f_2 &=\hat{ \theta}_n \left(
     \mu_{2n} -\mu_{1,n}^2\right) \,,  \\ 
    f_3  &=\frac{\hat{ \theta}_n}{2}   \left(\mu_{3,n}-\mu_{2,n} \mu_{1,n}-2\mu_{1,n}(\mu_{2,n}-\mu_{1,n}^2)\right) \,.  \end{split}\right. \label{e.3.11} 
\end{equation} 
Thus we have 
\begin{equation}
\left\{
\begin{split}& p\rho-(1-p) \xi =f_1 \\
& p\rho^2+(1-p)\xi^2 =f_2 \\
& p\rho^3-(1-p)\xi^3 =f_3 
\end{split}\right.  \label{e.3.12} 
\end{equation} 
The first equation in \eqref{e.3.12} yields
\begin{equation}
\xi=\frac{p\rho-f_1}{1-p}\,. \label{e.3.13} 
\end{equation}
Substituting to the second  equation in \eqref{e.3.12} 
we have 
\[
p\rho^2-2f_1p\rho+f_1^2-f_2(1-p)=0\,.
\]
Solving for $\rho$, we have 
\begin{equation}
\rho=\frac{f_1p\pm\sqrt{p(1-p)(f_2-f_1^2)}}{p}\,.
\label{e.3.14} 
\end{equation} 
Recalling  $\si=1$,  $\rho=\frac{1}{\eta}$ and $\xi=\frac{1}{\varphi}$    we have 
\[
f_2-f_1^2=p\rho^2+q\xi^2-(p\rho-q\xi)^2
=p(1-p)\rho^2+q(1-q)\xi^2+2pq\rho  \xi >0
\]
  so the discriminant defining  $\rho$ 
  (ie \eqref{e.3.14})  is  nonnegative.
Moreover, since 
$\xi=\frac{1}{\varphi}$, we see from \eqref{e.3.13}  that 
$\frac{p\rho-f_1}{1-p}=\frac{1}{\varphi}$ which means 
\[
\rho=\frac{f_1}{p}+\frac{1-p}{\varphi}>
f_1\,.
\]
Thus in \eqref{e.3.14}, we should take the positive sign to obtain
\begin{equation}
\rho=\frac{f_1p+\sqrt{p(1-p)(f_2-f_1^2)}}{p}\,.
\label{e.3.15} 
\end{equation}  
Now we substitute  $\xi$ given by \eqref{e.3.13}   into the third equation in \eqref{e.3.12} to obtain 
\[
 p \rho^3  -(1-p) \left(\frac{p\rho-f_1}{1-p}\right)^3=f_3\,. 
 \]
This means  
$$f_3(1-p)^2=p(1-p)^2\rho^3+(f_1-p\rho)^3\,.$$
Finally we substitute  $\rho$ in  the above equation 
by  \eqref{e.3.15}  to obtain one function equation for only one unknown   $p$: 
\begin{eqnarray}
&& (1-p)^2\left( f_1p+\sqrt{p(1-p)(f_2-f_1^2)}  \right) ^3\label{e.3.16} \\
&&\qquad  + p^2 \left(f_1-  f_1p - \sqrt{p(1-p)(f_2-f_1^2)} \right)^3-f_3p^2(1-p)^2=0\,. \nonumber
\end{eqnarray}
This equation depends on   $f_1, f_2, f_3$ 
computed from   the observation data of the 
double exponential Ornstein-Ulenbeck process. 
Although it is still hard to know if this function equation has a unique global solution or not. However, since it contain only one equation for one unknown we can plot the graph of the function (we consider the left hand side of \eqref{e.3.16} as a function  $g(p), 0< p< 1$) to see if $g(p)$ has a unique solution on the interval $0<p< 1$ or not. Or we can  plot  the derivative  
$g'(p)$ on $0< p< 1$ to see if $g'(p)$ remains the same sign or not.  If $g'(p)>0$ 
(or $g'(p)<0$) on $0< p< 1$,  then $g(p)=0$ has at most  one  solution on $(0, 1)$ by the mean value theorem.  

We   summarize the above discussions as the following theorem about the existence and uniqueness of the parameter estimators and their strong consistency results.
\begin{theorem}\label{t.3.1}  From the observation data, we denote 
$\mu_{k, n}$, $k=1, 2, 3, 4$  by \eqref{e.3.4}.
Then  $\hat \theta_n$ is given by \eqref{e.3.10},   namely
\begin{equation}
   \hat{ \theta}_n=\frac{1}{h}\ln \Big(\frac{\mu_{2,n}-\mu_{1,n}^2}{\mu_{4,n}-  \mu_{1,n} ^2}\Big)   
   \label{e.3.17} 
\end{equation} 
 and $f_k, k=1, 2, 3$ by 
\eqref{e.3.11}.  If \eqref{e.3.16} has a unique solution $\hat p_n$  on $(0, 1)$,  namely,
\begin{eqnarray}
&& (1-\hat p_n)^2\left( f_1\hat p_n+\sqrt{\hat p_n(1-\hat p_n)(f_2-f_1^2)}  \right) ^3\label{e.3.18} \\
&&\qquad  + \hat p_n^2 \left(f_1-  f_1\hat p_n - \sqrt{\hat p_n(1-\hat p_n)(f_2-f_1^2)} \right)^3-f_3\hat p_n^2(1-\hat p_n)^2=0  \nonumber
\end{eqnarray}
and if $\hat p_n$ is a continuous function of $f_1, f_2, f_3$,    then \eqref{e.3.6}-\eqref{e.3.9} has a unique solution $(\hat \th_n, \hat \xi_n, \hat \rho_n, \hat p_n)$ given by \eqref{e.3.17}, \eqref{e.3.18} and 
\begin{equation}
\left\{
\begin{split}
\hat \rho_n=&\frac{f_1\hat p_n+\sqrt{\hat p_n(1-\hat p_n)(f_2-f_1^2)}}{\hat p_n}\,, \\
\hat \xi_n=&\frac{\hat p_n\hat 
\rho_n-f_1}{1-\hat p_n}\,. 
\end{split}\right. \label{e.3.18} 
\end{equation}
Define 
\begin{equation}
 \hat \eta_n:=\frac1{\hat \rho_n}\,,\quad 
 \hat\varphi_n:= \frac1{\hat \xi_n}\,. 
\end{equation}
If $(\th, \eta, \varphi, p)$ are  the true parameters, namely, if the double exponential process $X_t$ satisfies \eqref{e.1.2} with the above parameters and with $\al=\sigma=1$, and 
if \eqref{e.3.16} has a unique solution when $f_1, f_2, f_3$ are replaced by their limits as $n\to \infty$,   then when $n\rightarrow \infty$,  $(\hat \th_n, \hat \eta_n, \hat \varphi_n, \hat p_n)
\to  (\th, \eta, \varphi, p)$    almost surely.   
\end{theorem}
\begin{proof}For any fixed $n$, it is clear that $f_1, f_2, f_3$ are continuous function of $\mu_{k, n}$, $k=1, 2, 3, 4$.  So, $\hat \th_n, \hat \xi_n, \hat \rho_n, \hat p_n$ are continuous functions of 
$\mu_{k, n}$, $k=1, 2, 3, 4$. Since $\mu_{k, n}$, $k=1, 2, 3, 4$ have limits as $n\to \infty$, we then see 
$(\hat \th_n, \hat \xi_n, \hat \rho_n, \hat p_n)$ have limits $(\hat  \th , \hat \xi , \hat \rho , \hat p) $. 
However, by  the above  argument, for each $n$,  
$\hat \th_n, \hat \xi_n, \hat \rho_n, \hat p_n$  satisfy \eqref{e.3.6}-\eqref{e.3.9}. Taking the limits of this system of equations we see 
 $(\hat  \th , \hat \xi , \hat \rho , \hat p)  $ satisfies 
\begin{equation}
\left\{\begin{split}
&\frac{1}{\hat \theta}\Big[\hat p\hat \rho-(1-\hat p)\xi\Big]
     =\lim_{n\to\infty}   \mu_{1,n}\,,  \\
&\frac{1}{\hat\theta}\Big[\hat p\hat \rho^2+(1-\hat p) \hat\xi^2\Big]=
   \lim_{n\to\infty} \left[  \mu_{2,n} -\mu_{1,n}^2 
   \right] \,,  \\
&    \frac{2 }{\hat \theta}\Big[\hat p
\hat \rho^3-(1-\hat p) \hat \xi^3\Big]
= \lim_{n\to\infty} \left[\mu_{3,n}-\mu_{2,n} \mu_{1,n}-2\mu_{1,n}(\mu_{2,n}-\mu_{1,n}^2)\right]\,,  \\  
    &\frac{1}{\hat \theta}e^{-\hat \theta h}\Big[\hat p\hat \rho^2+(1-\hat p) \hat \xi^2\Big]=\lim_{n\to\infty} \left[\mu_{4,n}-\mu_{1,n}^2\right] \,.  
\end{split}\right. \label{e.3.20}  
\end{equation} 
 Since \eqref{e.3.16} 
 has a unique solution   when $f_1, f_2, f_n$ are replaced by their limits as $n\to \infty$,  by the same argument as above we can show \eqref{e.3.20} has a unique solution. Obviously, $ (\th ,   \xi ,   \rho ,   p )$ satisfy \eqref{e.3.20}.  Thus $(\hat  \th , \hat \xi , \hat \rho , \hat p)= ( \th ,   \xi ,   \rho ,   p ) $. This means that   when $n\to\infty$, $(\hat \th_n, \hat \xi_n, \hat \rho_n, \hat p_n) \to (  \th ,   \xi ,  \rho ,   p)  $ almost surely and hence we obtain that   when $n\rightarrow \infty$,  $(\hat \th_n, \hat \eta_n, \hat \varphi_n, \hat p_n)
\to  (\th, \eta, \varphi, p)$    almost surely. 
 \end{proof}
\begin{remark} The estimators $(\hat \th_n, \hat \xi_n, \hat \rho_n, \hat p_n) $  defined in the above theorem are called the ergodic estimators of the parameters
$ (  \th ,   \xi ,  \rho ,   p) $.  The above theorem states that these ergodic estimators are uniquely determined and are strongly consistent. 
\end{remark} 

\section{Joint asymptotic behavior of all the obtained estimators}
In this section, we   shall prove the central limit theorem for  our ergodic estimators $\hat{\Theta}_n
=(\hat{\theta}_n,\hat{\eta}_n,\hat{\varphi}_n ,\hat{p}_n)$. Our  goal is to  prove that $\sqrt{n}(\hat{\Theta}_n-\Theta)$, where  $\Theta=( \theta,  \eta,\varphi, p)$  converges in law to  a mean zero normal vector and to find the asymptotic covariance matrix.  
Let 
\[  
\left\{\begin{split}
&g(x,y)= (g_1(x,y), g_2(x,y), g_3(x,y), 
g_4(x,y))^T\,,\\
& g_1(x,y)= x,\quad  g_2(x,y)=x^2, \quad g_3(x,y)=x^3,\quad g_4(x,y)=xy
 \end{split}\right. 
 \]
 and
 \[
\mu=(\mu_1,\mu_2,\mu_3,\mu_4)
 \,,\quad\hbox{where}\quad \mu_k=\EE[g_j(\XX_o,\XX_h)],\quad k=1,2,3,4\,.
 \] 
Denote
\[
\mu_n=(\mu_{1, n}\,,  \mu_{2, n}\,,\mu_{3, n}\,,\mu_{4, n})\,,
\]
where $\mu_{k,n}$,  $k=1,2,3,4$ are defined by 
\eqref{e.3.4}. 

First, we have the following  central limiting result. 
\begin{lemma}\label{l.4.1}  Let $  \mu_n$,  $\mu$ 
and $g$ be defined as above.  Then as $n\to\infty$, we have 
\begin{equation}
  \sqrt{n}(\mu_n-\mu)\xrightarrow[]{d}N(0, A)\,,   
\end{equation}
with  the $4\times 4$  covariance matrix $A$  being given by 
\begin{equation}
A=\left( \sigma_{g_ig_j}\right)_{1\le i,j\le 4}\,, 
\end{equation}
where $  \sigma_{g_ig_j}$, $1\le i,j\le 4$ will be given in the appendix. 
\end{lemma}
\begin{proof}
We shall use the  Cramer-Wold device (e.g. \cite[Theorem 29.4]{Billingsley}).  For any $a=(a_1, a_2, a_3, a_4)^T\in \RR^4$, 
consider  $ a^T \mu_n=\sum_{k=1}^4 a_k \mu_{k, n}$.  
By   \cite[Theorem 2.6]{Mas07} and \cite{Van},   the double exponential Ornstein-Uhlenbeck  process $\{X_t\}$ is  exponentially $\beta$-mixing. By the fact that  the exponential  $\beta$-mixing implies the  exponential  $\alpha$-mixing and  by  
the central limit theorem 
(e.g. \cite[Theorem 18.6.2]{IL71}) for stationary process with exponential $\alpha$-mixing, we have
\begin{equation}
\sqrt{n} a^T (\mu_n-\mu)\xrightarrow[]{d}\mathcal{N}(0,\sigma_a^2)\,, 
\end{equation}  
Since $a\in \RR^4$ is arbitrary, we prove the lemma 
through the Cramer-Wold device. 
\end{proof}
Denote
\[
\left\{\begin{split}
&h_1(  \th ,   \xi ,  \rho ,   p) =\frac{1}{\theta}\Big[p\rho-(1-p) \xi\Big] \,, \\
&h_2(  \th ,   \xi ,  \rho ,   p)=\frac{1}{\theta}\Big[p\rho^2+(1-p)\xi^2\Big]  \,,   \\
&  h_3(  \th ,   \xi ,  \rho ,   p)=  \frac{1 }{\theta}\Big[p\rho^3-(1-p)\xi^3\Big] \,, \\  
    &h_4(  \th ,   \xi ,  \rho ,   p)=\frac{1}{\theta}e^{-\theta h}\Big[p\rho^2+(1-p)\xi^2\Big] \,.
  \end{split}\right. 
\] 
and 
\[
\left\{\begin{split}
&\tilde h_1(\mu_1, \mu_2, \mu_3,,\mu_4)
     =\mu_1\,,  \\
&\tilde h_2(\mu_1, \mu_2, \mu_3,\mu_4)
=    \mu_2 -\mu_1^2 \,,   \\
&\tilde h_3(\mu_1, \mu_2, \mu_3,\mu_4)=   \mu_{3 }-\mu_{2 } \mu_{1 }-2\mu_{1 }(\mu_{2 }-\mu_{1 }^2)\,,  \\  
    &\tilde h_4(\mu_1, \mu_2, \mu_3,\mu_4)= \mu_{4 }-\mu_{1 }^2\,. 
     \end{split}\right. 
\]  
Set
\[
h=(h_1, h_2, h_3, h_4)^T \quad\hbox{and}\quad 
\tilde h=(\tilde h_1, \tilde h_2, \tilde h_3, \tilde  h_4)^T
\,.
\]
We compute the partial derivative of $\tilde h$ with respect to $\mu$ to obtain
\begin{align*}
&\frac{ \partial{\tilde h_1}}{\partial{\mu_1}}=1\,,\quad 
\frac{ \partial{\tilde h_1}}{\partial{\mu_2}}=\frac{ \partial{\tilde h_1}}{\partial{\mu_3}}=\frac{ \partial{\tilde h_1}}{\partial{\mu_4}}=0\\    
&\frac{ \partial{\tilde h_2}}{\partial{\mu_1}}=-2\mu_1\,,\quad \frac{ \partial{\tilde h_2}}{\partial{\mu_2}}=1\,,\quad
\frac{ \partial{\tilde h_1}}{\partial{\mu_3}}=\frac{ \partial{\tilde h_1}}{\partial{\mu_4}}=0\\  
&\frac{ \partial{\tilde h_3}}{\partial{\mu_1}}=-3\mu_2-6\mu_1^2\,,\quad \frac{ \partial{\tilde h_3}}{\partial{\mu_2}}=-3\mu_1\,,\quad\frac{ \partial{\tilde h_3}}{\partial{\mu_3}}=1\,,\quad \frac{ \partial{\tilde h_3}}{\partial{\mu_4}}=0\\
&\frac{ \partial{\tilde h_4}}{\partial{\mu_1}}=-2\mu_1\,,\quad \frac{ \partial{\tilde h_4}}{\partial{\mu_4}}=1\,,\quad \frac{ \partial{\tilde h_4}}{\partial{\mu_2}}=\frac{ \partial{\tilde h_4}}{\partial{\mu_3}}=0
\end{align*}
We compute  the partial derivatives of $h$ with respect to the parameters to obtain 
\begin{align*}
&\frac{ \partial{h_1}}{\partial{p}}=\frac{1}{\theta}(\rho+\xi)\,,\quad 
\frac{ \partial{h_2}}{\partial{p}}=\frac{1}{\theta}(\rho^2-\xi^2)\,,\quad 
\frac{ \partial{h_3}}{\partial{p}}=\frac{1}{\theta}(\rho^3+\xi^3)\\
&\frac{ \partial{h_4}}{\partial{p}}=\frac{1}{\theta}e^{-\theta h}(\rho^2-\xi^2)\,,\quad 
\frac{ \partial{h_1}}{\partial{\rho}}=\frac{1}{\theta}(p-q\xi)\,,\quad 
\frac{ \partial{h_2}}{\partial{\rho}}=\frac{1}{\theta}(2p\rho+q\xi^2)\\
&\frac{ \partial{h_3}}{\partial{\rho}}=\frac{1}{\theta}(3p\rho^2-q\xi^3)\,,\quad 
\frac{ \partial{h_4}}{\partial{\rho}}=\frac{1}{\theta}e^{-\theta h}(2p\rho+q\xi^2)\,,\quad 
\frac{ \partial{h_1}}{\partial{\xi}}=\frac{1}{\theta}(p\rho-q)\\
&\frac{ \partial{h_2}}{\partial{\xi}}=\frac{1}{\theta}(p\rho^2+2q\xi)\,,\quad 
\frac{ \partial{h_3}}{\partial{\xi}}=\frac{1}{\theta}(p\rho^3-3q\xi^2)\,,\quad 
\frac{ \partial{h_4}}{\partial{\xi}}=\frac{1}{\theta}e^{-\theta h}(p\rho^2+2q\xi)\\
&\frac{ \partial{h_1}}{\partial{\theta}}=\frac{-1}{\theta}(p\rho-q\xi)\,,\quad 
\frac{ \partial{h_2}}{\partial{\theta}}=\frac{-1}{\theta}(p\rho^2+q\xi^2)\,,\quad 
\frac{ \partial{h_3}}{\partial{\theta}}=\frac{-1}{\theta}(p\rho^3-q\xi^3)\\
&\frac{ \partial{h_4}}{\partial{\theta}}=\frac{-1}{\theta}e^{-\theta h}(p\rho^2+q\xi^2)\Big[\frac{1}{\theta}+1\Big]
\end{align*}
Let us denote the matrix
\[
\nabla_\Theta h(\Theta)=
\begin{pmatrix}
\frac{\partial h_1}{\partial p } &\frac{\partial h_1}{\partial \rho } & \frac{\partial h_1}{\partial \xi } &\frac{\partial h_1}{\partial \theta }\\
\frac{\partial h_2}{\partial p }&\frac{\partial h_2}{\partial \rho }  & \frac{\partial h_2}{\partial \xi }  &\frac{\partial h_2}{\partial \theta }\\
\frac{\partial h_3}{\partial p } &\frac{\partial h_3}{\partial \rho }  &  \frac{\partial h_3}{\partial \xi } &\frac{\partial h_3}{\partial \theta }\\
\frac{\partial h_4}{\partial p } &\frac{\partial h_4}{\partial \rho }  &  \frac{\partial h_4}{\partial \xi } &\frac{\partial h_4}{\partial \theta }
\end{pmatrix}
\]	

Then we have the following result.
\begin{theorem}
Denote $\Theta=( \theta , \eta,\varphi, p )$  and $\hat{\Theta}_n
=(\hat{\theta}_n ,\hat{\eta}_n,\hat{\varphi}_n,\hat{p}_n)$. If $\hat p_n$ is a continuous function of $f_1, f_2, f_3$ and  
if \eqref{e.3.16} has a unique solution when $f_1, f_2, f_3$ are replaced by their limits as $n\to \infty$,   then   as $n\rightarrow\infty$ we have  
\begin{equation}
\sqrt{n}(\hat \Theta_n -\Theta)\xrightarrow[]{d}\mathcal{N}(0, \Sigma )
\end{equation}
where 
\begin{equation}
 \Sigma =\left(\left(\nabla h\right)^{-1} \nabla \tilde h
 \right)^T  A \left(\nabla h\right)^{-1} \nabla \tilde h\,. 
\end{equation}
\end{theorem}
\begin{proof}
It is easy to see that  $h, \tilde h:\RR^4\to \RR^4$ defined as above are smooth mappings. Using these two mappings, we can write the system 
\eqref{e.3.6}-\eqref{e.3.9} to determine the ergodic estimators $\Theta_n$ 
\begin{equation}
h(\Theta_n) =\tilde h(\mu_n)\,. 
\end{equation}
From Theorem \ref{t.3.1}, it follows that 
$h$ has inverse $h^{-1}$ so that
\[
 \Theta_n  = (h^{-1}\circ \tilde h)(\mu_n)\,. 
\]
By Lemma \ref{l.4.1} and the Delta method, we
see that 
\begin{equation}
\sqrt{n}(\hat \Theta_n -\Theta)\xrightarrow[]{d}\mathcal{N}(0, \Sigma )
\end{equation}
where 
\begin{align*}
 \Sigma =
 &(\nabla_\mu (h^{-1}\circ \tilde h))^T  A  \nabla_\mu (h^{-1}\circ \tilde h)\\
 =& \left(\left(\nabla h\right)^{-1} \nabla \tilde h
 \right)^T  A \left(\nabla h\right)^{-1} \nabla \tilde h\,. 
\end{align*}
This proves the theorem. 
\end{proof}

\section{Exact Simulation for the double exponential Ornstein-Uhlenbeck  process}
Before we give some numerical simulations to validate our ergodic estimators, in this section we propose  a distributional decomposition to  exactly  simulate the double exponential Ornstein-Uhlenbeck  process.   We follow the idea  of   \cite{Qu}, where  the  exact simulation of Gamma Ornstein-Uhlenbeck  process is studied. First, we have the following result. Without loss of generality we can assume $\si=1$. 
\begin{theorem}Let $X_t$ be the   double exponential Ornstein-Uhlenbeck  process given by \eqref{e.1.2}.  
For any $t, t_1>0$, the  Laplace transform of $X_{t+t_1}$ conditioning  on $X_t$ is given  by 
\begin{equation}
\begin{split}
\EE[e^{i uX_{t+t_1}}|X_t]=e^{-iuwX_t}\exp\Big[\frac{-\lambda p}{\theta}\int_{0}^{\infty}(1-e^{-ius})\int_{1}^{1/w}\eta v e^{-s\eta v}\frac{1}{v}dvds\\
-\frac{\lambda q}{\theta}\int_{-\infty}^0(1-e^{-ius})\int_{1}^{1/w}\phi v e^{s\phi v}\frac{1}{v}dvds\Big]\,, 
\end{split}\label{e.5.4} 
\end{equation}
where   $w=e^{-\theta t_1}$. 
\end{theorem}
\begin{proof}
Recall the $\Psi$ defined by \eqref{e.2.4} and the formula \eqref{e.2.5}. 
We can write the characteristic 
function of  $X_{t_1}=\si \int_t^{t+t_1} e^{-(t+t_1-s)}
dZ_s$ as 
\begin{equation}
\begin{split}
\EE[e^{i uX_{t_1}}]=&\exp\Big[\int_t^{t+t_1}\Psi(\sigma e^{-\theta (t+t_1-s)}u)ds\Big]
\\
=&\exp\Big[\int_0^{t_1}-\lambda \Big(1-\EE\Big(e^{(i\sigma e^{-\theta ({t_1}-s)}uY_1)}\Big)\Big)ds\Big]\,. 
\end{split}
\end{equation}
Denote  $ \hat{h}(z)=\EE\Big(e^{ zY_1 }\Big) $.   
The   Laplace transform of $X_{t+t_1}$ conditioning  on $X_t$ is  
\begin{equation}
\begin{split}
\EE[e^{i uX_{t+t_1}}|X_t]=&e^{-iuwX_t}\exp\Big[\int_t^{t+t_1}-\lambda \Big(1-\hat{h}(i\sigma e^{-\theta (t+t_1-s)}u)\Big)ds\Big]\\
=&e^{-iuwX_t}\exp\Big[\int_0^{t_1}-\lambda \Big(1-\hat{h}(i\sigma e^{-\theta s}u)\Big)ds\Big]\,. 
\end{split}
\end{equation}
Let $ue^{-\theta t_1}=x$, then for $\sigma=1$, we have 
\begin{align*}
\int_0^{t_1}\Big(1-\hat{h}(iu e^{-\theta s})\Big)ds
=
&\frac{1}{\theta}\int_{u}^{ue^{-\theta t_1}}\frac{(1-\hat{h}(ix))}{x}dx\\
=&I_1+I_2\,,
\end{align*} 
where 
\begin{align*}
I_1=&\frac{1}{\theta}\int_{uw}^{u}\frac{1}{x}\int_0^{\infty}(1-e^{-ixy})\eta p e^{-\eta y}dydx\,, 
\\
I_2=& \frac{1}{\theta}\int_{uw}^{u}\frac{1}{x}\int_{-\infty}^0(1-e^{-ixy})\phi q e^{\phi y}dydx\,. 
\end{align*} 
The first term $I_1$   can be written 
\begin{align*}
I_1=&\frac{1}{\theta}\int_{0}^{\infty}\frac{(1-e^{-ius})}{s}\int_s^{s/w}\eta p e^{-\eta y}dyds\\
=&\frac{p}{\theta}\int_{0}^{\infty}\frac{(1-e^{-ius})}{1} \frac{e^{-\eta s}-e^{-\eta s/w}}{s}dy\\
=&\frac{p}{\theta}\int_{0}^{\infty}\frac{(1-e^{-ius})}{1}\int_{\eta}^{\eta/w}e^{-sv}dvds\\
=&\frac{p}{\theta}\int_{0}^{\infty}\frac{(1-e^{-ius})}{1}\int_{1}^{1/w}\eta v e^{-s\eta v}\frac{1}{v}dvds\,.
\end{align*} 
The second term $I_2$  can be written as 
\begin{align*}
I_2=& \frac{1}{\theta}\int_{-\infty}^0\frac{(1-e^{-ius})}{s}\int_s^{s/w}\phi q e^{-\eta y}dyds\\
=&\frac{q}{\theta}\int_{-\infty}^0\frac{(1-e^{-ius})}{1} \frac{e^{\phi s}-e^{\phi s/w}}{s}dy\\
 =& \frac{q}{\theta}\int_{-\infty}^0\frac{(1-e^{-ius})}{1}\int_{\phi}^{\phi/w}e^{sv}dvds\\
 =&\frac{q}{\theta}\int_{-\infty}^0\frac{(1-e^{-ius})}{1}\int_{1}^{1/w}\phi v e^{s\phi v}\frac{1}{v}dvds\,.
\end{align*} 
This gives us \eqref{e.5.4}, 
proving the theorem. 
\end{proof}
Since the second exponential factor on the right hand side of \eqref{e.5.4} is the characteristic function of the  compound Poisson process   
we have
\begin{corollary}[Exact Simulation via Decomposition Approach]  Let $N$ be  a Poisson random variable of rate $\lambda h$
  and let  $\{S_k\}_{k=1,2,\ldots}$  be  i.i.d random variables following a mixture of double exponential distribution 
 \begin{equation}
\begin{split}
f_{S_k}(y)=&  p\eta e^{\theta h U} e^{-\eta e^{\theta h U}y}I_{y\geq 0}+q\phi e^{\theta h U}e^{\phi e^{\theta h U} y}I_{y< 0}\,, \\
 &\qquad   \hspace{5mm}\forall\  k=1\,,2\,,  \ldots\,,  
 \end{split} 
    \end{equation} 
where  $U\overset{d}{=} \mathcal{U}[0,1] $ is the uniform distribution on $[0, 1]$.     Then
    \begin{equation}
    X_{t+h}\overset{d}{=}X_{t}e^{-\theta h}+\sum_{k=1}^{N} S_k \,. \label{e.5.6} 
    \end{equation} 
\end{corollary}

The above formula \eqref{e.5.6}  enables us to simulate the process $X_t$ by the exact decomposition approach. 
%
%

\section{Numerical results}
To validate our estimators discussed in Section 4, we perform some numerical simulations.  We choose the  values of $p=0.6$, $\eta=1.2$, $\varphi=1.6$ and $\theta=2.0$ (and $\la=\si=1$). With these parameters,  we simulate the double exponential Ornstein-Uhlenbeck process   using the exact decomposition algorithm given by \eqref{e.5.6}. 
A simulated sample is displayed in 
Figure \ref{fig:1}.  Figures \ref{fig:2} and
\ref{fig:3} plot the assumed values versus the values by the ergodic estimators. Table 1 lists the approximation of ergodic estimators to the true parameters as the time becomes larger. It demonstrates that the rate of convergence is quite faster. 
\begin{figure}[h!]
\centering
\includegraphics[scale=0.25]{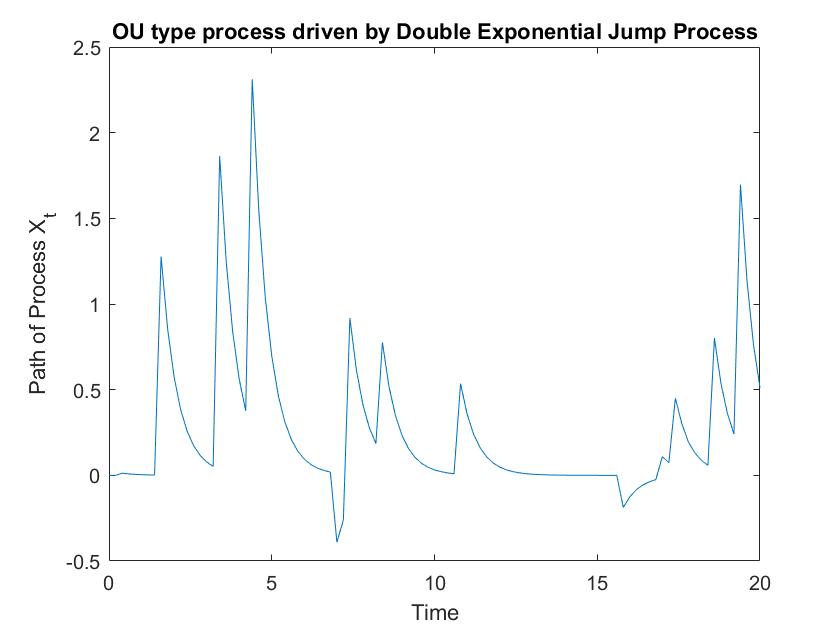}
\caption{Simulated sample path for a double exponential Orntein-Uhlenbeck process with T=20, N=50, h=0.02 $\eta=1.2$, $\varphi=1.6$ and $\theta=2.0$, $\sigma=\lambda=1$}
\label{fig:1}
\end{figure}
\newpage
The table \ref{table:1} shows the estimated values of the parameters $p$, $\eta$, $\phi$ and $\theta$ with different number of steps N and fixed $h=0.02$ and $T=Nh$
\begin{figure}%
\centering
\subfloat[\centering Estimated Values of p]
{{\includegraphics[width=5cm]{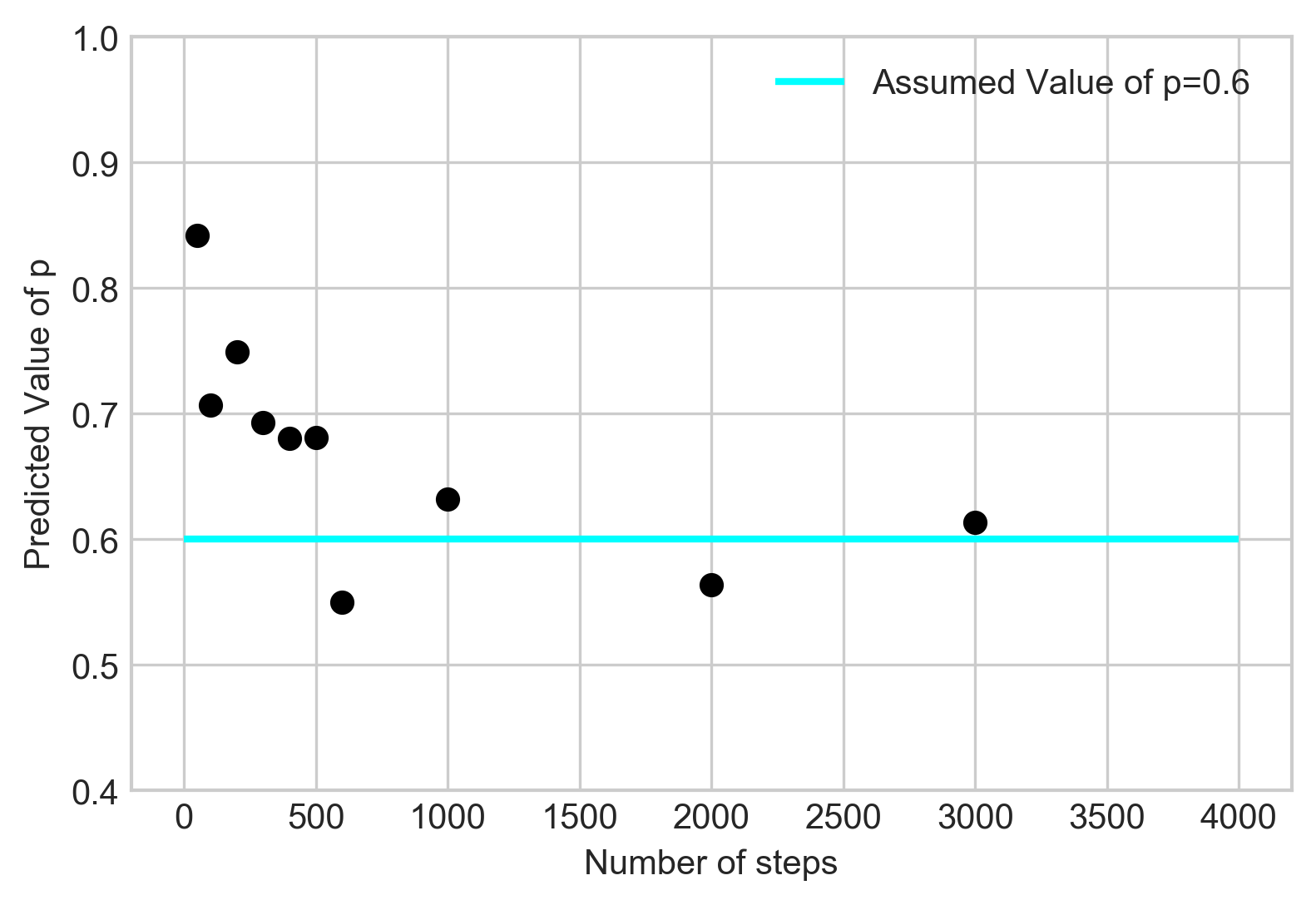} }}%
\qquad
\subfloat[\centering Estimated Values of $\eta$]
{{\includegraphics[width=5cm]{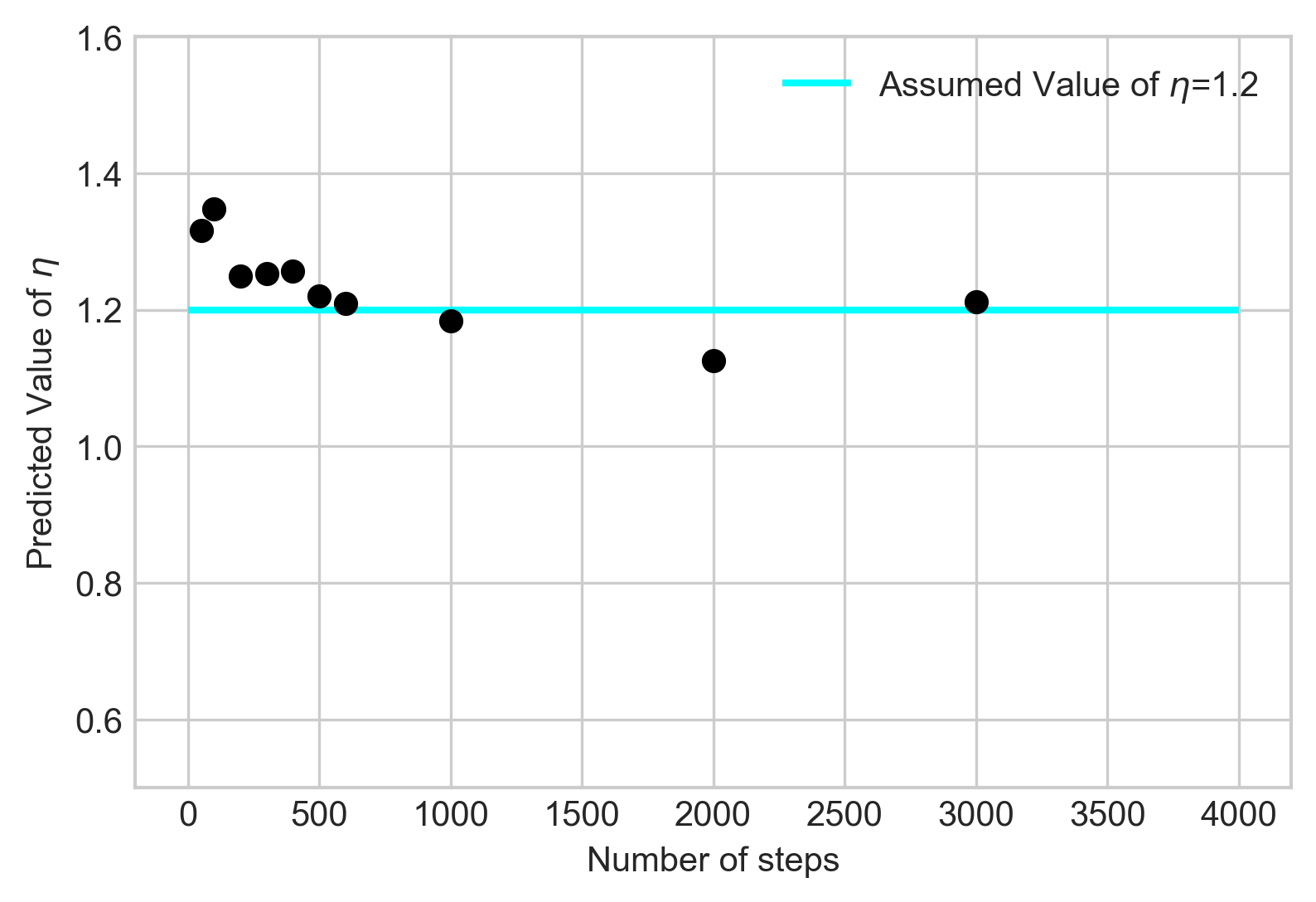} }}%
\caption{Assumed versus estimated values}%
\label{fig:example}%
\label{fig:2}
\end{figure}
\begin{figure}%
\centering
\subfloat[\centering Estimated Values of $\varphi$]
{{\includegraphics[width=5cm]{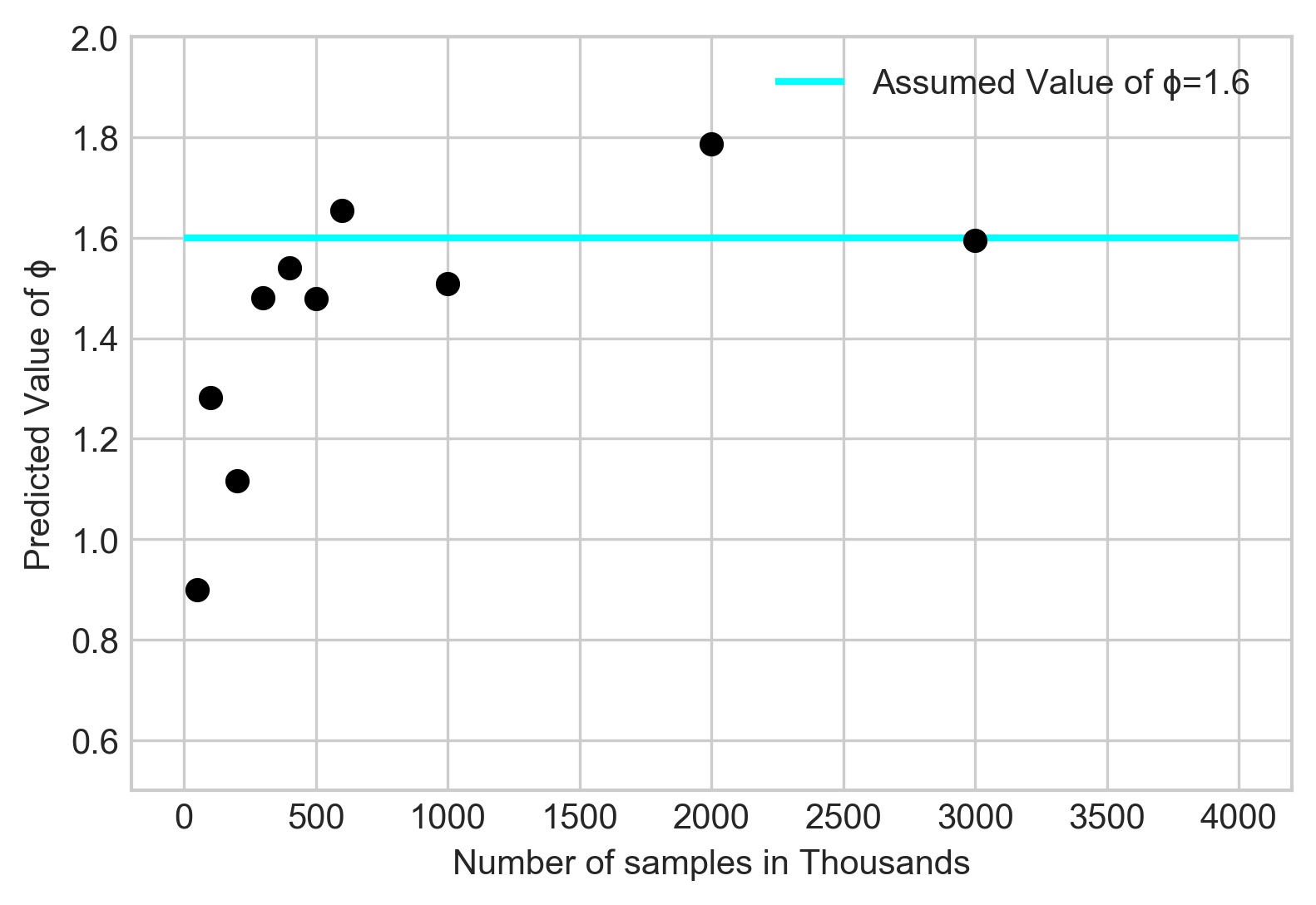} }}%
\qquad
\subfloat[\centering Estimated Values of $\theta$]
{{\includegraphics[width=5cm]{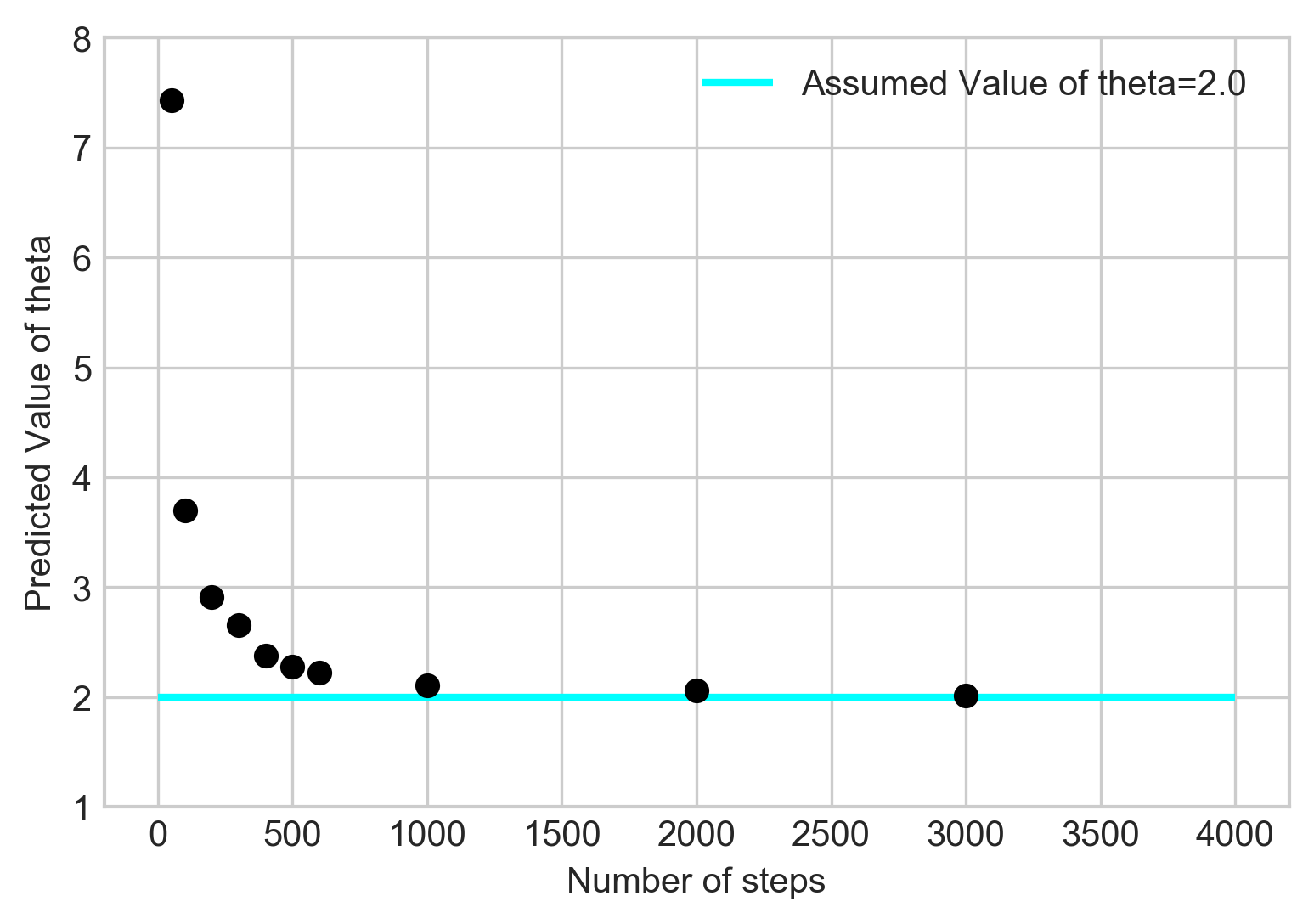} }}%
\caption{Assumed versus estimated values}%
\label{fig:3}
\label{fig:example}%
\end{figure}
\begin{table}[h!]
\centering
\begin{tabular}{||c c c c c c||} 
 \hline
  Time&Number of steps & $p=0.6$& $\eta=1.2$& $\phi=1.6$& $\theta=2.0$ \\ [0.5ex] 
 \hline\hline
 1&50&0.8421& 1.31677& 0.8995& 7.4297\\
 2&100&0.7070& 1.3477& 1.2816& 3.6995\\
 4&200&0.74925& 1.2498& 1.1164& 2.9127\\
 6&300&0.6928& 1.2532& 1.4803& 2.6587 \\
 8&400& 0.6804& 1.2571& 1.5397& 2.3808\\
 10&500& 0.6812& 1.2204& 1.4793& 2.2743\\
 12&600&0.5500& 1.2089& 1.6546& 2.2217\\
 20&1000& 0.6320& 1.1836& 1.5078& 2.1066\\
 40&2000&0.5635& 1.1255& 1.7866& 2.0631\\
 60&3000&0.6135& 1.2112& 1.5940& 2.0128\\ [1ex] 
 \hline
\end{tabular}
\caption{Assumed Values and Estimated values of the parameters with different number of steps N and fixed $h=0.02$ and $T=Nh$}
\label{table:1}
\end{table}

\section{Appendix:   Covariance matrix $A$}
In this section we give the expression of the covariance  matrix $A$ in Lemma \ref{l.4.1}. It is very sophisticated to express the entries of this matrix in terms of the parameters of the equation \eqref{e.2.1}. So, we keep them as expression of the invariant probability measure of $\XX_0$ and that of $\XX_{kh}$. First, we compute $\sigma_{g_1g_1}$. 
\begin{equation}
\begin{split}
\sigma_{g_1g_1} &=\cov(\XX_{0},\XX_{0}) +2\sum_{j=1}^\infty[\cov(\XX_{0},\XX_{jh})]   \\
&=\EE(\XX_o^2)-\EE(\XX_o)^2+2\sum_{j=1}^\infty\Big[\EE(\XX_o{\XX_{jh}})-\EE(\XX_o)\EE({\XX_{jh}})\Big]\\
&=\EE(\XX_o^2)-\EE(\XX_o)^2+2\sum_{j=1}^\infty\Big[\EE(\XX_o{\XX_{jh}})-\left[\EE(\XX_o)\right]^2 \Big]\,,  
\end{split}
\end{equation}
where we used $\EE({\XX_{jh}})=\EE({\XX_{0}})$. 
Now we compute $\sigma_{g_2g_2}$. 
\begin{equation}
\begin{split}
\sigma_{g_2g_2} &=\cov(\XX_{0}^2,\XX_{0}^2) +2\sum_{j=1}^\infty[\cov(\XX_{0}^2,\XX_{jh}^2)]   \\
&=\EE(\XX_o^4)-\EE(\XX_o^2)^2+2\sum_{j=1}^\infty\Big[\EE(\XX_o^2{\XX_{jh}}^2)-\EE(\XX_o^2)^2\Big]\,. 
\end{split}
\end{equation}
Similarly, we have 
\begin{equation}
\begin{split}
    \sigma_{g_3g_3} &=\cov(\XX_{0}^3,\XX_{0}^3) +2\sum_{j=1}^\infty[\cov(\XX_{0}^3,\XX_{jh}^3)]   \\
&=\EE(\XX_o^6)-\EE(\XX_o^3)^2+2\sum_{j=1}^\infty\Big[\EE(\XX_o^3{\XX_{jh}}^3)-\EE(\XX_o^3)^2\Big]
\end{split}
\end{equation}
and 
\begin{equation}
\begin{split}
    \sigma_{g_4g_4} &=\cov(\XX_{0}\XX_{h},\XX_{0}\XX_{h}) +2\sum_{j=1}^\infty[\cov(\XX_{0}\XX_{h},\XX_{jh}\XX_{(j+1)h})]   \\
&=\EE((\XX_{0}\XX_{h})^2)-\EE(\XX_{0}\XX_{h})^2\\
&\qquad +2\sum_{j=1}^\infty\Big[\EE(\XX_{0}\XX_{h}\XX_{jh}\XX_{(j+1)h})-\EE(\XX_{0}\XX_{h})\EE(\XX_{jh}\XX_{(j+1)h})\Big]\,. 
\end{split}
\end{equation}
$\sigma_{g_1g_2}$ is computed as follows. 
\begin{equation}
\begin{split}
    \sigma_{g_1g_2} &=\cov(\XX_{0},\XX_{0}^2) +\sum_{j=1}^\infty[\cov(\XX_{0},\XX_{jh}^2)+\cov(\XX_{0}^2,\XX_{jh})]   \\
&=\EE((\XX_{0})^3)-\EE(\XX_{0})\EE(\XX_{0}^2)+\sum_{j=1}^\infty\Big[\EE(\XX_{0}\XX_{jh}^2)-\EE(\XX_{0})\EE(\XX_{jh}^2)\\
&+\EE(\XX_{0}^2\XX_{jh})-\EE(\XX_{0}^2)\EE(\XX_{jh})\Big]\,. 
\end{split}
\end{equation}
In similar way we can get 
\begin{equation}
\begin{split}
    \sigma_{g_1g_3} &=\cov(\XX_{0},\XX_{0}^3) +\sum_{j=1}^\infty[\cov(\XX_{0},\XX_{jh}^3)+\cov(\XX_{0}^3,\XX_{jh})]   \\
&=\EE((\XX_{0})^4)-\EE(\XX_{0})\EE(\XX_{0}^3)+\sum_{j=1}^\infty\Big[\EE(\XX_{0}\XX_{jh}^3)-\EE(\XX_{0})\EE(\XX_{jh}^3)\\
&+\EE(\XX_{0}^3\XX_{jh})-\EE(\XX_{0}^3)\EE(\XX_{jh})\Big]
\end{split}
\end{equation}
and 
\begin{equation}
\begin{split}
\sigma_{g_1g_4} &=\cov(\XX_{0},\XX_{0}\XX_{h}) +\sum_{j=1}^\infty[\cov(\XX_{0},\XX_{jh}\XX_{(j+1)h})+\cov(\XX_{jh},\XX_{0}\XX_{h})]   \\
&=\EE(\XX_{0}^2\XX_{h})-\EE(\XX_{0})\EE(\XX_{0}\XX_{h})+\sum_{j=1}^\infty\Big[\EE(\XX_{0}\XX_{jh}\XX_{(j+1)h})-\EE(\XX_{0})\EE(\XX_{jh}\XX_{(j+1)h})\\
&+\EE(\XX_{0}\XX_{h}\XX_{jh})-\EE(\XX_{jh})\EE(\XX_{0}\XX_{h})\Big]\,. 
\end{split}
\end{equation}
 $\sigma_{g_2g_3}$ is similar to $ \sigma_{g_1g_2}$.
\begin{equation}
\begin{split}
    \sigma_{g_2g_3} &=\cov(\XX_{0}^2,\XX_{0}^3) +\sum_{j=1}^\infty[\cov(\XX_{0}^2,\XX_{jh}^3)+\cov(\XX_{0}^3,\XX_{jh}^2)]   \\
&=\EE((\XX_{0})^5)-\EE(\XX_{0}^2)\EE(\XX_{0}^3)+\sum_{j=1}^\infty\Big[\EE(\XX_{0}^2\XX_{jh}^3)-\EE(\XX_{0}^2)\EE(\XX_{jh}^3)\\
&+\EE(\XX_{0}^3\XX_{jh}^2)-\EE(\XX_{0}^3)\EE(\XX_{jh}^2)\Big]\,. 
\end{split}
\end{equation}
Finally, we have 
\begin{equation}
\begin{split}
\sigma_{g_2g_4} &=\cov(\XX_{0}^2,\XX_{0}\XX_{h}) +\sum_{j=1}^\infty[\cov(\XX_{0}^2,\XX_{jh}\XX_{(j+1)h})+\cov((\XX_{jh})^2,\XX_{0}\XX_{h})]   \\
&=\EE(\XX_{0}^3\XX_{h})-\EE(\XX_{0}^2)\EE(\XX_{0}\XX_{h})+\sum_{j=1}^\infty\Big[\EE(\XX_{0}^2\XX_{jh}\XX_{(j+1)h})-\EE(\XX_{0}^2)\EE((\XX_{jh}\XX_{(j+1)h})\\
&+\EE(\XX_{0}\XX_{h}\XX_{jh}^2)-\EE(\XX_{jh}^2)\EE(\XX_{0}\XX_{h})\Big]
\end{split}
\end{equation}
and 
\begin{equation}
\begin{split}
\sigma_{g_3g_4} &=\cov(\XX_{0}^3,\XX_{0}\XX_{h}) +\sum_{j=1}^\infty[\cov(\XX_{0}^3,\XX_{jh}\XX_{(j+1)h})+\cov((\XX_{jh})^3,\XX_{0}\XX_{h})]   \\
&=\EE(\XX_{0}^4\XX_{h})-\EE(\XX_{0}^3)\EE(\XX_{0}\XX_{h})+\sum_{j=1}^\infty\Big[\EE(\XX_{0}^3\XX_{jh}\XX_{(j+1)h})-\EE(\XX_{0}^3)\EE((\XX_{jh}\XX_{(j+1)h})\\
&+\EE(\XX_{0}\XX_{h}\XX_{jh}^3)-\EE(\XX_{jh}^3)\EE(\XX_{0}\XX_{h})\Big]\,. 
\end{split}
\end{equation}

\end{document}